\documentclass[10pt,a4paper]{amsart}

\usepackage[utf8]{inputenc}
\usepackage{enumerate}
\usepackage{amsmath}
\usepackage{amsfonts}
\usepackage{amsthm}
\usepackage{amssymb}
\usepackage{mathtools}
\usepackage{bbm}
\usepackage[mathscr]{eucal} 
\usepackage{breqn}

\usepackage{hyperref}
\usepackage{esint}
\usepackage{fullpage}
\usepackage[foot]{amsaddr}
\usepackage{latexsym}
\usepackage{xcolor}

\usepackage{physics}
\usepackage[style=numeric, sorting=nyt,giveninits=true,url=false,maxnames=99]{biblatex}
\numberwithin{equation}{section}
\newtheorem{theorem}{Theorem}[section]
\newtheorem{corollary}[theorem]{Corollary}
\newtheorem{lemma}[theorem]{Lemma}
\newtheorem{proposition}[theorem]{Proposition}
\theoremstyle{definition}
\newtheorem{definition}[theorem]{Definition}

\newtheorem{remark}[theorem]{Remark}
\newtheorem{example}[theorem]{Example}

\newcommand{\HH}{\mathrm H}
\newcommand{\LL}{\mathrm L}
\newcommand{\CC}{\mathrm C}
\def\ZO{\mathcal{Z}}
\def\ZOO{\mathcal{Z}^\Omega}
\def\LLL{\mathcal{L}} 
\def\RR{\mathbb{R}} 
\def\half{\frac12}

\title{A class of port-Hamiltonian systems with general time delays}

\author{Rogelio Arancibia}
\email{arancibiabustos@uni-wuppertal.de}
\author{B\'alint Farkas}
\email{farkas@uni-wuppertal.de}

\author{Birgit Jacob}
\email{bjacob@uni-wuppertal.de}

\author{Merlin Schmitz}
\email{meschmitz@uni-wuppertal.de}

\address{School of Mathematics and Natural Sciences, University of Wuppertal, Gau\ss stra\ss e 20, 42119 Wuppertal, Germany}

\begin{document}

\begin{abstract}
  We define a class of infinite dimensional linear port-Hamiltonian delay systems, give a characterization for such systems, and provide an easily verifiable sufficient condition for a system to belong to this class, thereby extending the study initiated in \cite{breiten2024towards} by Breiten, Hinsen and Unger. The results rely on the operator semigroup theoretic treatment of linear delay PDEs.
\end{abstract}

\thanks{This work was funded by the Deutsche Forschungsgemeinschaft (DFG, German Research Foundation) – Project-ID 531152215 – CRC 1701.}

\maketitle

The study of delay differential equations (or generally functional-differential equations) is a challenging and still very active field of the research that goes back, not surprisingly, to the works of Leonhard Euler, Pierre-Simon Laplace, and to many of their contemporaries. Nowadays, there are several well established approaches to study existence, uniqueness of solutions, and stability or other qualitative/quantitative properties of these. The literature is extensive, and we mention only a narrow selection of monographs bearing direct relevance to this paper. An operator semigroup theoretic treatment can be found in the books \cite{batkai2005semigroups} by B\'atkai and Piazzera, and \cite{CuZw20} by Curtain and Zwart, the latter bringing a more general perspective of applications and also focusing on concrete problems. A standard reference placing delay equations in the context of linear system theory, via system nodes, is the monograph  \cite{Staffans} by Staffans. How the theory of evolutionary equations can be applied to a particular class of delay equations in the port-Hamiltonian setting is described in the recent manuscript \cite{AignerWaurick} by Aigner and Waurick (see Section 6.8 there), and we further refer to the monograph \cite{SeiTroWau22} by Seifert, Trostorff, and Waurick, for the theory of evolutionary equations at large.  The treatise \cite{Fri14} by Fridman gives a thorough introduction to many system theoretic key aspects of delay equations. While these works focus on delay equations with finite time horizon, the case of infinite time horizon is well documented, too, see the lecture notes \cite{HinMurNai91} by Hino, Murakami, and Naito, the paper by Hale and Kato \cite{HalKat78} or the book \cite{BenDapDel07} by Bensoussan, Da Prato, Delfour, and  Mitter which also focuses on important system  theoretic aspects.

While the port-Hamiltonian paradigm provides an extremely efficient, powerful and general framework for the study of interconnected systems, see, e.g., the survey \cite{vdSJel} by van der Schaft and Jeltsema, it is pointed out in \cite{SchFriOrtRai16} by Fridman, Schiffer, Ortega, and Raisch that delay terms appear to cause problems when it comes to port-Hamiltonian modeling or Lyapunov function techniques. Motivated by applications to microgrids the paper \cite{SchFriOrtRai16} then presents conditions for the stability for a class of finite dimensional port-Hamiltonian  systems with delays. Passivity of certain systems of semilinear diffusion equations including distributed and time-dependent discrete delays, even with infinite time horizon, is studied by Solomon and Fridman in \cite{SolFri15}. It is important to note that the underlying state space there is \emph{infinite dimensional}, but the arguments are attached to the properties of the main operator governing the diffusion (such as the Dirichlet Laplacian in case of cubes).

The present paper provides a systematic study  of the following  class of abstract, \emph{infinite dimensional}, linear delay equations with \emph{general delay terms}
\begin{equation*}
    \begin{cases}
        \dot{x}(t)= \displaystyle A_0 Q_0 x(t)+ \Phi(x(t+\cdot))+Bu(t),
        &t \ge 0,  \\
        x(t)=f(t),                         &  t\in [-\tau,0 ], 
        \\
        y(t)=  B^\ast Q_0  x(t), & t\ge 0,  
    \end{cases}
\end{equation*}
where $A_0$ is a maximal dissipative operator on a Hilbert space, $\Phi$ is the delay operator and $B$ is a linear bounded operator between Hilbert spaces. For further information concerning this class of systems we refer to Section \ref{sec:intro}. While Niculescu and Lozano in \cite{NicLoz01} studied the passivity of abstract, linear, finite dimensional delay equations with single point delay, in \cite{breiten2024towards} Breiten, Hinsen and Unger made a step towards a more general class of finite  dimensional (linear, ODE) 
port-Hamiltonian delay systems, see Section \ref{sec:disc}.

The authors of \cite{breiten2024towards} gave a definition of such a system being port-Hamiltonian along with a characterization of this property. A similar result, and an extension, for the case of a single point delay can also be found in \cite{Kur24} by Kurula as an example for time-delay port-Hamiltonian systems; there the approach is based on system nodes. El Haskouki and Gernandt take up studying  the  class described in \cite{breiten2024towards} from the viewpoint of delay (in)dependence of structural properties, such as passivity or stability, see \cite{ElhGer26}.

This paper generalizes the work \cite{breiten2024towards} in various ways: We allow infinite dimensional state spaces, opening the door for applications to delay port-Hamiltonian PDEs; we allow more general delay terms than the finite linear combination of point delays, in particular we can cover  distributed delays or combination of these; essentially any complex/signed finite   Radon measure, or even more general expressions, can appear in the delay term.

Here we have chosen a different formulation of the delay system than the authors in \cite{breiten2024towards}, see \eqref{PHDDE} below. The framework here is more suitable for infinite dimensional state spaces and fits extremely well to the operator theoretic treatment of delay equations with $\LL^p$-phase space and operator semigroup techniques, see the papers \cite{BaPiPaper,BaPiPaper2} or the monograph \cite{batkai2005semigroups} by B\'atkai and Piazzera. Our definition for a port-Hamiltonian delay system will be formally different than the one given by Breiten, Hinsen and Unger in \cite{breiten2024towards}  (see Definition \ref{def:phs}), but at the end the two definitions turn out to be equivalent in the case of finite-dimensional state spaces and finitely many point delays (that is for systems that were considered in \cite{breiten2024towards}). The comparison between the two settings is given in Example \ref{exa:breitencompare} below. As a special case of our approach, we discuss the situation of infinitely many point delays (over a finite time horizon) and give a characterization of port-Hamiltonian delay systems with specific Hamiltonian in the spirit of the mentioned paper \cite{breiten2024towards}. Finally, we provide an easily verifiable sufficient condition: Informally speaking, if the undelayed part shows strong enough dissipation as compared to the delay terms, one can always choose the Hamiltonian in a way that the system becomes port-Hamiltonian. In fact, a variant of this result has already appeared in Webb's paper \cite{webb1976} (with different terminology), whose relevance in this context has not been noticed until now; we shall comment on this in Remark \ref{rem:Webb}. In the first section, we collect known results concerning delay equations that are needed for the investigation of 
port-Hamiltonian delay systems in Section \ref{sec:2}. The paper is concluded in Section \ref{sec:disc} with a  comparison of the results in \cite{breiten2024towards} and here, along with the study of more general discrete delay operators.

Extension to delay operators with infinite time horizon, to non-linear equations, the study of more general Hamiltonians, or delay (in)dependence of structural properties, is the subject of future research.

\section{Notation and preliminaries on delay equations}
\label{sec:intro}
In this paper, we assume that $H$ is a (real or complex) Hilbert space with scalar product $\langle\cdot,\cdot\rangle$, $A_0\colon D(A_0)\to H$ is a linear operator with domain $D(A_0)$. The space of bounded linear operators on $H$ is denoted by $\mathcal{L}(H)$. Let us fix $\tau>0$ for describing the time horizon of the delay.
 Further, we assume that $Q_0 \in \mathcal{L}(H)$ is positive, self-adjoint (in notation: $Q_0\geq 0$), and invertible. In this case also $Q_0^{-1}$ is positive and self-adjoint.

Next, we describe the delay operator. Let $\eta\colon [-\tau,0]\rightarrow \mathcal{L}(H)$ be a function of bounded variation.   We define $\Phi\colon \CC([-\tau,0]; H)\rightarrow H$ by
    \begin{equation}\label{eq:PhiRS}
        \Phi f \coloneqq \int_{-\tau}^0 \dd\eta f,
    \end{equation}
    where the integral is calculated as a Riemann-Stieltjes integral. Indeed, operator $\Phi$ is well-defined on the Banach space  $\CC([-\tau,0]; H)$ of $H$-valued continuous functions, and maps into $H$. It is obviously linear and bounded with operator norm satisfying 
$\|\Phi\|\le \mathrm{Var}(\eta)$.

 \begin{example}[Discrete delay operator]\label{exa:discdelay}
Let $(A_k)_{k\in \mathbb N}$ be a sequence in $\mathcal{L}(H)$ with $(\|A_k\|)_{k\in \mathbb N}\in \ell^1$ (here and below $\ell^1$ denotes the Banach space of absolutely summable sequences), and let $0<\tau_1<\tau_2<\ldots<\tau_k<\cdots<\tau $ with $\tau\coloneqq \lim_{k\rightarrow \infty}\tau_k$.  The operator $\Phi$ defined by
  \begin{equation*}
        \Phi f \coloneqq \sum_{k=1}^\infty A_k f(-\tau_k), \quad f\in \CC([-\tau, 0];H),
    \end{equation*}
is of the form given in \eqref{eq:PhiRS}. 
Indeed, with
\begin{equation*}
\eta\coloneqq \sum_{k=1}^\infty A_k \mathbf{1}_{(-\tau_k,0]},
\end{equation*}
where $\mathbf{1}_{(-\tau_k,0]}$ is the characteristic function of the given interval, we have
\begin{equation*}
        \Phi f = \int_{-\tau}^0 \dd\eta f.
    \end{equation*}
    It is illustrative to write the Riemann-Stieltjes measure here as $\dd\eta \coloneqq \sum_{k=1}^\infty A_k \delta_{-\tau_k}$, where $\delta_{-\tau_k}$ is the point evaluation (Dirac measure) at $-\tau_k$. 
\end{example}

After having given the structural ingredients we can present the class of delay equations we are concerned with in this paper as follows: 
\begin{equation}
    \begin{cases}
        \dot{x}(t)= \displaystyle A_0 Q_0 x(t)+ \Phi(x(t+\cdot))+Bu(t),
        &t \ge 0,  \\
        x(t)=f(t),                         &  t\in [-\tau,0 ], 
        \\
        y(t)=  B^\ast Q_0  x(t), & t\ge 0,  
    \end{cases}
    \label{PHDDE}
\end{equation}
where $A_0\colon D(A_0)\subset H \rightarrow H$ is maximal dissipative,  $B \in \mathcal{L}(U,H)$ with $B^\ast$ denoting its adjoint,  $U$ is a Hilbert space with inner product $\langle\cdot,\cdot\rangle_U$; the input space. Here $u$ denotes the input function and $f$ the initial value.

For a suitable reformulation of \eqref{PHDDE} it is useful to introduce for a  given function $v\colon [-\tau,\infty)\to H$ the so-called \emph{history functions}  $v_t$ defined as $v_t(\sigma) \coloneqq v(t+\sigma)$  for $t\geq 0$ and $\sigma\in [-\tau,0]$. In this way, the delay term in \eqref{PHDDE} takes the form $\Phi x_t$.

\begin{definition}
  Let   
  $u\in \CC([0,\infty); U)$ and  $f\in \CC([-\tau,0]; H)$. A function $x\colon [-\tau,\infty)\rightarrow H$ is called a \emph{classical solution of \eqref{PHDDE}} if $Q_0 x(t) \in D (A_0)$ for all $t \geq 0$, $x_t\in \HH^1([-\tau,0];H)$ for every $t\geq0$,  $x|_{[0,\infty)}\in  \CC^1([0,\infty);H)$ and \eqref{PHDDE} is satisfied for all $t\ge 0$.
\end{definition}

Let us briefly summarize how delay equations of the above form can be treated by operator theoretic methods. In order to show the existence of classical solutions and to characterize port-Hamiltonian delay systems, we reformulate \eqref{PHDDE} as an abstract Cauchy problem.
Keeping port-Hamiltonian systems in mind a  natural choice for the state space is $\ZO \coloneqq H \times \LL^2 ( [-\tau ,0] ; H )$, which is a Hilbert space with the natural inner product
\begin{dmath*}
\left\langle \begin{bmatrix}
    x_1  \\
    f_1
\end{bmatrix} , \begin{bmatrix}
    x_2\\
    f_2
\end{bmatrix}  \right\rangle_{\ZO} {:=}  \langle x_1,  x_2\rangle+\langle f_1,f_2\rangle_{\LL^2}= \langle x_1,  x_2\rangle   + \int_{-\tau } ^0 \langle f_1 (s), {f_2}(s)\rangle   \dd s.
\end{dmath*}

We define the operators 
$\mathcal{A}\colon D(\mathcal{A}) \subset \ZO \to \ZO$, $\mathcal{B} \colon U \to \ZO$ and $\mathcal{Q} \colon  \ZO\to \ZO$ by 
\begin{align}
        \mathcal{A} \begin{bmatrix}
            x\\
            f
        \end{bmatrix}&\coloneqq \begin{bmatrix}
            A_0 & \Phi\\
            0 & \frac{\dd}{\dd s}
        \end{bmatrix} \begin{bmatrix}
            x\\
            f
        \end{bmatrix} 
        \label{phgenerator},\\
 D(\mathcal{A})&\coloneqq \Biggl\{ \begin{bmatrix}
            x\\
            f
        \end{bmatrix} \in \ZO \: \Big|\:f \in \HH^{1}( [-\tau , 0] ;H) \text{ and } f(0)=x\in D(A_0) \Biggr\},
        \label{phdomaingenerator}  \\\nonumber
        \mathcal{B} u &\coloneqq \begin{bmatrix}
        Bu \\
        0
    \end{bmatrix},\\\nonumber
    \mathcal{Q} \begin{bmatrix}
            x\\
            f
        \end{bmatrix} &\coloneqq \begin{bmatrix}
            Q_0 & 0 \\
            0  & I
        \end{bmatrix} \begin{bmatrix}
            x\\
            f
        \end{bmatrix} =
        \begin{bmatrix}
        Q_0x \\
         f
    \end{bmatrix}.
\end{align}
We remark that $\mathcal Q$ is self-adjoint, positive and invertible on $\ZO$. Here $\HH^1([-\tau,0];H)$ denotes the first Sobolev space of $H$-valued functions. Note that 
$\HH^1([-\tau,0];H)$ is continuously embedded into $\CC([-\tau,0];H)$ so that the coupling condition $f(0)=x$ in \eqref{phdomaingenerator} is meaningful and also $\Phi$ can be applied to any $f\in \HH^1([-\tau,0];H)$.
Thus, the delay system \eqref{PHDDE} has the following operator representation:
\begin{equation}
    \begin{cases}
        \dot{z}(t) = \mathcal{A Q} z(t) + \mathcal{B} u(t), & t\geq0,\\
        z(0) =z_0,\\
       y(t)= \mathcal{B^\ast Q}z(t),& t\geq 0,
    \end{cases}
    \label{AbstractODE2}
\end{equation}
with $z(t)= \left[\begin{smallmatrix} x(t)\\ x_t
    \end{smallmatrix}\right]$ and $z_0= \left[\begin{smallmatrix} f(0)\\ f
    \end{smallmatrix}\right]$.

In this abstract setting the following solution concepts  are important.    
\begin{definition}
Let $u\in \CC^1([0,\infty);U)$ and $\mathcal{Q}z_0\in  D(\mathcal{A})$. A function $z\colon[0,\infty)\rightarrow \ZO$ is called a \emph{classical solution for \eqref{AbstractODE2}} if $\mathcal{Q}z(t)\in  D(\mathcal{A})$ for all $t\ge 0$, $z \in \CC^1([0,\infty);\ZO)$ and \eqref{AbstractODE2} is satisfied for all $t\ge 0$.

Let $u\in \LL^2([0,\infty);U)$ and $z_0\in \ZO $. A function $z\colon[0,\infty)\rightarrow \ZO$ is called a \emph{mild solution for \eqref{AbstractODE2}} if $z \in \CC([0,\infty);H)$ and
\begin{equation*}
z(t) = z_0 + \mathcal{A} \int_0^t  \mathcal{Q}z(s) \dd s + \int_0^t  \mathcal{B}u(s) \dd s, \qquad t\ge 0.
\end{equation*}
\end{definition}

We have the following relation between classical solutions of the delay system \eqref{PHDDE} and classical solutions of \eqref{AbstractODE2}, see, e.g., \cite[Proposition 2.2]{BaPiPaper2} or Corollary 3.1.4 and  Proposition 3.9 \cite{batkai2005semigroups} (more precisely, their respective proofs in the mentioned references).

\begin{theorem}\label{thm:equiv}
If $\left[\begin{smallmatrix} v\\ w\end{smallmatrix}\right]$ is a classical solution of \eqref{AbstractODE2}, then $w(t)=x_t$ for every $t\geq 0$, and $x:[-\tau,\infty)\to H$ defined as
\[
x(t):=\begin{cases} v(t),&t\geq0,\\
f(t),&t\leq 0,
\end{cases}
\]
is a classical solution of \eqref{PHDDE}. Conversely, if $x$ is a classical solution of \eqref{PHDDE}, then $z(t):=\left[\begin{smallmatrix} x(t)\\ x_t\end{smallmatrix}\right]$ (for $t\geq 0)$ defines a classical solution of \eqref{AbstractODE2}.
\end{theorem}

The statement in the next theorem  about the existence of solutions  follows from \cite[Theorem 3.23]{batkai2005semigroups}, and the standard theory of inhomogeneous linear Cauchy problems, see, e.g., \cite[Section VI.7]{engel2000one} or \cite[Section 4.2]{Pazy}.
  \begin{theorem}\label{the:BaPi}
The operator $\mathcal{A}$ defined in \eqref{phgenerator} with domain given by \eqref{phdomaingenerator} generates a $C_0$-semigroup in $\ZO$. Further, for every $u\in \CC^1([0,\infty);U)$ and $\mathcal{Q}z_0\in  D(\mathcal{A})$, system \eqref{AbstractODE2} possesses a unique classical solution, and for every $u\in \LL^2([0,\infty);U)$ and $z_0\in  \ZO$,  system \eqref{AbstractODE2} has  a unique mild solution. 
    \end{theorem}

Altogether, we see that for every $f\in \LL^2([-\tau,0];H)$ with  $f\in \HH^1([-\tau,0];H)$ and $Q_0f(0)\in D(A_0)$ there is a unique mild solution to \eqref{PHDDE}.

\section{Port-Hamiltonian delay systems}\label{sec:2}
We turn to the definition of port-Hamiltonian delay systems. Our aim is to give a consistent extension of the definition given by Breiten, Hinsen and Unger in \cite{breiten2024towards}.

 In the following, we assume that $\Omega\in \LL^\infty([-\tau,0]; \mathcal{L}(H))$ and that $\Omega(t)$ is positive, self-adjoint for each $t\in[-\tau,0]$.
Consider the sesquilinear form (bilinear form, in case of real Hilbert spaces)
\begin{equation*} 
\left\langle \begin{bmatrix}
    x_1  \\
    f_1
\end{bmatrix} , \begin{bmatrix}
    x_2\\
    f_2
\end{bmatrix}  \right\rangle_{\ZOO} \coloneqq   \langle x_1,  x_2\rangle   + \int_{-\tau } ^0 \langle f_1 (s), \Omega(s)f_2(s)\rangle   \dd s.
\end{equation*}
Note that that the previous definition makes sense and the sesquilinear form is symmetric (or Hermitian), positive semi-definite and  continuous with respect to the norm on $\mathcal{Z}$. Note also if $\Omega$ is bounded from below, then the sesquilinear form $\langle\cdot,\cdot\rangle_{\ZOO}$ is, in fact, an inner product defining an equivalent norm on $\mathcal{Z}$.

We equip the delay system \eqref{PHDDE} with the Hamiltonian $\mathcal{H}\colon \CC([-\tau,\infty);H)\times [0,\infty)\rightarrow \mathbb R$ given by
\begin{equation}\label{eq:Ham}
\mathcal{H}(x,t) \coloneqq  \frac{1}{2}\left\langle\begin{bmatrix}
    x(t)  \\
    x_t
\end{bmatrix} , \begin{bmatrix}
    Q_0 & 0\\ 0& I  
\end{bmatrix}\begin{bmatrix}
    x(t)\\
         x_t
\end{bmatrix}  \right\rangle_{\ZOO}.
\end{equation}

\begin{definition}
\label{def:phs}
We call the delay system \eqref{PHDDE} with Hamiltonian \eqref{eq:Ham} a \emph{port-Hamiltonian delay system} if
\begin{equation}\label{eq:dissi}
    \frac{\dd}{\dd t}\mathcal{H}(x,t) \le 
    \Re \langle u(t), y(t)\rangle_{U}
\end{equation}
for every classical solution of \eqref{PHDDE}.
\end{definition}

\begin{proposition}\label{prop:pHS}
The  system \eqref{PHDDE} with Hamiltonian \eqref{eq:Ham} is a port-Hamiltonian delay system if and only if $\mathcal{A}$  is dissipative with respect to the sesquilinear form $\langle\cdot,\cdot\rangle_{\ZOO}$ on $\mathcal{Z}$, i.e.,
\begin{equation}\label{eq:calAdiss}
 \Re   \left\langle \begin{bmatrix}
    x_0  \\
    f_0
\end{bmatrix} , \mathcal{A}\begin{bmatrix}
    x_0\\
     f_0
\end{bmatrix}  \right\rangle_{\ZOO}  \le 0\quad\text{for every $\begin{bmatrix}
    x_0\\
     f_0
\end{bmatrix}\in D(\mathcal{A})$.}
\end{equation}

\end{proposition}

\begin{proof}
 We first assume that  $\mathcal{A}$ has the dissipativity property \eqref{eq:calAdiss}. Let $x$ be a classical solution of \eqref{PHDDE}.  Then $\left[\begin{smallmatrix} x\\ x_t\end{smallmatrix}\right]$ is a classical solution of \eqref{AbstractODE2} by Theorem \ref{thm:equiv}. Thus
  $ \mathcal{Q}\left[\begin{smallmatrix}
    x(t)\\
     x_t
\end{smallmatrix}\right]\in D( \mathcal{A})$ and  we calculate
\begin{align*}
\frac{\dd}{\dd t}\mathcal{H}(x,t) &=  \frac{1}{2}\frac{\dd}{\dd t}\left\langle \begin{bmatrix}
    x(t)  \\
    x_t
\end{bmatrix} , \mathcal{Q}\begin{bmatrix}
    x(t)\\
     x_t
\end{bmatrix}  \right\rangle_{\ZOO}
\\
& = \Re \left\langle \frac{\dd}{\dd t}\begin{bmatrix}
    x(t)  \\
    x_t
\end{bmatrix} , \mathcal{Q}\begin{bmatrix}
    x(t)\\
     x_t
\end{bmatrix}  \right\rangle_{\ZOO}\\
 &= \Re \left\langle \mathcal{A}\mathcal{Q}\begin{bmatrix}
    x(t)  \\
    x_t
\end{bmatrix} +\mathcal{B}u(t), \mathcal{Q}\begin{bmatrix}
    x(t)\\
     x_t
\end{bmatrix}  \right\rangle_{\ZOO}  \\
&\leq \Re \langle B u(t) , Q_0 x(t) \rangle\\
&= \Re \langle u(t), y(t)\rangle_{U}.
\end{align*}  
Conversely, we assume that the system \eqref{PHDDE} with Hamiltonian \eqref{eq:Ham} is a port-Hamiltonian delay system. Let $\mathcal{Q}\left[\begin{smallmatrix}
    x_0\\
     f_0
\end{smallmatrix}\right]\in D( \mathcal{A})$ and $u=0$. Then system \eqref{PHDDE} has a unique classical solution $x$ and we calculate
\begin{align*}
  \Re   \left\langle \mathcal{Q}\begin{bmatrix}
    x_0  \\
    f_0
\end{bmatrix} , \mathcal{A}\mathcal{Q}\begin{bmatrix}
    x_0\\
     f_0
\end{bmatrix}  \right\rangle_{\ZOO} =\frac{\dd}{\dd t}\mathcal{H}(x,t)|_{t=0} \le 0.
\end{align*}
Thus  $\mathcal{A}$ is dissipative with respect to the sesquilinear form $\langle \cdot, \cdot \rangle_{\mathcal{Z}^\Omega}$, which concludes the proof.
\end{proof}

We will present a characterization for system  \eqref{PHDDE} to be a port-Hamiltonian delay system below in Theorem \ref{thm:pHs2}.
We first formulate some lemmata. The first is an integration by parts formula for Banach space valued functions, which is recalled from the literature.

\begin{lemma}[Theorem 2.3A in \cite{Naralenkov}]\label{lem:ibp1}
	Let $E,F,X$ be Banach spaces and let $\bullet\colon E\times F\to X$ be a continuous bilinear mapping. Let $f\colon [a,b]\to E$, $g\colon [a,b]\to F$ be functions. Then
		\begin{equation*}
		\int_a^b f(s)\bullet \dd g(s)+\int_a^b \dd f(s)\bullet g(s)=f(b)\bullet g(b)-f(a)\bullet g(a),
		\end{equation*}
		provided one, and then both, of these integrals exist. The integrals are to be understood as Riemann-Stieltjes integrals.
	\end{lemma} 

\begin{lemma}[Integration by parts for bilinear forms on $\CC^1$]\label{lem:ibp2}
	Let $H$ be a Hilbert space, let  $f,g\colon [a,b]\to H$, and let $\Omega\colon [a,b]\to \LLL(H)$. Suppose $\Omega$ is of bounded variation and $f,g$ are continuously differentiable. Then
\begin{dmath*}
	\int_a^b \langle f(s),(\dd\Omega(s)) g(s)\rangle+ \int_a^b \langle  f'(s),\Omega(s) g(s)\rangle\dd s+ \int_a^b \langle  f(s),\Omega (s)g'(s)\rangle\dd s
	=\langle f(b),\Omega(b)g(b)\rangle-\langle f(a),\Omega(a)g(a)\rangle.
	\end{dmath*}
\end{lemma}
\begin{proof} Without of loss of generality we may assume that $H$ is a real Hilbert space (i.e., its scalar product is bilinear). For given $S\in \LLL(H)$
	the bilinear mapping  $H\times H\to \mathbb{R}$, $(u,v)\mapsto \langle u,Sv\rangle$ is 
    bounded, hence factorizes through the projective tensor product $E\coloneqq H\Hat{\otimes}_\pi H$, i.e., there is a unique bounded linear functional $\Hat S\colon E\to\mathbb{R}$ such that $\langle u,Sv\rangle=\Hat S(u\otimes v)$ for every $u,v\in H$, see, e.g., \cite[Theorem 2.9]{Ryan}.
    Then the mapping $\bullet\colon E\times \LLL(H)\to \RR$, $(x,S)\mapsto   \Hat S(x)$ is bilinear and bounded. Moreover, it satisfies
	\begin{equation*}
	\bullet(u\otimes v,S)=\langle u,Sv\rangle,
	\end{equation*}
so that the first assumption in Lemma \ref{lem:ibp1} is satisfied if we choose $F=\LLL(H)$ and $X=\mathbb{R}$.

Let $f,g\colon [a,b]\to H$ be continuously differentiable functions. Then $f\otimes g\colon t\mapsto f(t)\otimes g(t)$ is continuously differentiable and $(f\otimes g)'=f'\otimes g+f \otimes g'$. From Lemma \ref{lem:ibp1} we obtain that
\begin{dmath*}
\int_{a}^b (f\otimes g)(s)\bullet \dd\Omega(s)+\int_{a}^b (f\otimes g)'(s)\bullet \Omega(s)\dd s=(f\otimes g)(b)\bullet \Omega(b)-(f\otimes g)(a)\bullet \Omega(a).
\end{dmath*}
But this is precisely the assertion.
\end{proof}

\begin{lemma}[Integration by parts for bilinear forms on $\HH^1$]\label{lem:ibp3}
		Let $H$ be a Hilbert space, let  $f,g\in \HH^1([a,b];H)$ and let $\Omega\colon [a,b]\to \LLL(H)$ be of bounded variation. Then 
	\begin{dmath*}
		\int_a^b \langle f(s),(\dd\Omega(s)) g(s)\rangle+ \int_a^b \langle  f'(s),\Omega(s) g(s)\rangle\dd s+ \int_a^b \langle  f(s),\Omega (s)g'(s)\rangle\dd s
		\qquad=\langle f(b),\Omega(b)g(b)\rangle-\langle f(a),\Omega(a)g(a)\rangle.
		\end{dmath*}
	\end{lemma}
	\begin{proof}
	Take sequences $(f_n)$, $(g_n)$  of continuously differentiable functions with $f_n\to f$ and $g_n\to g$ in $\HH^1([a,b];H)$. Then we have $f_n\to f$ and $g_n\to g$ 	pointwise and uniformly (by the continuous embedding of $\HH^1$ in $\CC$). By Lemma \ref{lem:ibp2} we have
	\begin{dmath*}
		\int_a^b \langle f_n(s),(\dd\Omega(s)) g_n(s)\rangle+ \int_a^b \langle  f_n'(s),\Omega(s) g_n(s)\rangle\dd s+ \int_a^b \langle  f_n(s),\Omega (s)g_n'(s)\rangle\dd s
		=\langle f_n(b),\Omega(b)g_n(b)\rangle-\langle f_n(a),\Omega(a)g_n(a)\rangle.
		\end{dmath*}
		By passing to the limit and using the convergence properties asserted above, we obtain the statement.
\end{proof}

Our main result that gives a full characterization of port-Hamiltonian delay system reads as follows.

\begin{theorem}\label{thm:pHs2}
Suppose, besides the standing assumptions, that $\Omega\colon [-\tau,0]\rightarrow \mathcal{L}(H)$  is of bounded variation. Then
 the system  \eqref{PHDDE} with Hamiltonian \eqref{eq:Ham} is a port-Hamiltonian delay system if and only if 
 \begin{dmath*}
  \Re\left( \langle f(0),  \left(  A_0 +\frac{1}{2} \Omega(0) \right)f(0)\rangle - \frac{1}{2}\langle f(-\tau),  \Omega(-\tau) f(-\tau)\rangle   
              -
            \frac{1}{2}\int_{-\tau } ^0 \langle f,(\dd\Omega) f\rangle
            +      \langle f(0),\Phi f\rangle\right) \le 0 \end{dmath*}
 for every $f \in \HH^{1}( [-\tau , 0] ; H)$ with $f(0)\in D(A_0)$. 
\end{theorem}

\begin{proof}  By Proposition \ref{prop:pHS} the  system  \eqref{PHDDE} with Hamiltonian \eqref{eq:Ham} is a port-Hamiltonian delay system if and only if the operator $\mathcal{A}$  is dissipative with respect to the sesquilinear form $\langle \cdot,\cdot\rangle_{\ZOO}$.
For $\left[\begin{smallmatrix}
            x\\
            f \end{smallmatrix} \right]\in D(\mathcal A)$ we calculate, by using Lemma \ref{lem:ibp3}:
\begin{align*}
\Re \left\langle \mathcal{A}  \begin{bmatrix}x\\f \end{bmatrix},  \begin{bmatrix}x\\ f \end{bmatrix}  \right\rangle_{\ZOO} 
=&\Re\left\langle \begin{bmatrix} A_0 & \Phi\\ 0 & \frac{\dd}{\dd t} \end{bmatrix} \begin{bmatrix}    x\\  f \end{bmatrix},  \begin{bmatrix}   x\\    f \end{bmatrix}  \right\rangle_{\ZOO}\\
   =&  \Re  \langle x,A_0 x\rangle  +  \Re \langle x, \Phi f\rangle+ \Re\int_{-\tau } ^0 \langle f' (s), \Omega(s) f(s)\rangle\dd s \\ 
   =&  \Re  \langle x,A_0 x\rangle  + 
            \Re\frac{1}{2} \langle x, \Omega(0) x\rangle 
            - \Re\frac{1}{2}\langle f(-\tau),  \Omega(-\tau) f(-\tau)\rangle\\
            &-  \Re
            \frac{1}{2}\int_{-\tau } ^0 \langle f(s),(\dd \Omega(s)) f(s)\rangle 
            +   \Re  \langle x, \Phi f\rangle.
\end{align*}
After substituting $x=f(0)$, this concludes the proof.
\end{proof}

Theorem \ref{thm:pHs2} provides an equivalent condition guaranteeing that the system \eqref{PHDDE} with Hamiltonian \eqref{eq:Ham} is a port-Hamiltonian delay system. The next theorem shows that the function $\Omega$ can be chosen suitable such that system \eqref{PHDDE} with Hamiltonian \eqref{eq:Ham} is a port-Hamiltonian delay system provided $A_0$ is a quasi-dissipative with constant at least  $- |\dd \eta|([-\tau,0])$.  Here $|\dd\eta|$ denotes the total variation (measure) of $\dd \eta$. 

\begin{theorem}\label{thm:webbmini}
Assume that 
    $\Re \langle A_0 x, x \rangle \le - |\dd \eta|([-\tau,0]) \|x\|^2$ for every $x\in H$.
Then  there exists a  function $\Omega\colon [-\tau,0]\rightarrow \mathcal{L}(H)$  of bounded variation such that $\Omega(t)$ is positive, self-adjoint for every $t\in[-\tau,0]$ and such that system \eqref{PHDDE} is a port-Hamiltonian system with Hamiltonian \eqref{eq:Ham} with state space $\ZO$.
A possible choice for $\Omega$ is given by 
 $\Omega =|\dd\eta| I.$
\end{theorem}

\begin{remark}\label{rem:Webb}
    In  \cite{webb1976} Webb poses the condition
    \begin{equation}\label{eq:webbeta}
    \lim_{t\to -\tau}|\dd \eta|([-\tau,t])\neq 0.
    \end{equation}
    This yields for $\Omega(t)\coloneqq |\dd\eta|([-\tau,t])I$ that $\Omega$ is bounded away from $0$, and the sesquilinear form  $\langle\cdot,\cdot\rangle_{\ZOO}$ is in fact an inner product defining an equivalent norm on $\ZO$. Webb then proves that the strong dissipativity $\Re \langle A_0 x, x \rangle\le \alpha\|x\|^2$  with some $\alpha\le -|\dd \eta|([-\tau,0])$ for all $x\in D(A_0)$, implies the dissipativity of $\mathcal{A}$ (the other parameters in Webb's theorem are: $p=q=2$, $\beta=1$). This amounts, via the Lumer-Philipps theorem, to the generator property of $\mathcal{A}$, thus the well-posedness of \eqref{AbstractODE2} and the existence of classical solutions as in Theorem \ref{the:BaPi}. The condition \eqref{eq:webbeta} is indispensable for this kind of arguments. For the existence of classical solutions here we relied on the B\'atkai-Piazzera result, Theorem \ref{the:BaPi}, which is based on perturbation techniques instead of dissipativity arguments. Therefore we do not require the extra condition \eqref{eq:webbeta} about $\eta$. In this way, Theorem \ref{thm:webbmini} is slightly more general than Proposition 4.1 in \cite{webb1976} (note however that the proofs are essentially the same).
\end{remark}
\begin{proof}
As said, we choose $\Omega(s) =|\dd\eta|([-\tau,s])I$. 
Then with $\Omega(0)=|\dd \eta|([-\tau,0])I$ we conclude 
for $f \in \HH^{1}( [-\tau , 0] ; H)$  that
\begin{align*}
  \Re&\left( \langle f(0),  \left(  A_0 +\frac{1}{2} \Omega(0) \right)f(0)\rangle - \frac{1}{2}\langle f(-\tau),  \Omega(-\tau) f(-\tau)\rangle   
              -
            \frac{1}{2}\int_{-\tau } ^0 \langle f,(\dd\Omega) f\rangle
            +      \langle f(0),\Phi f\rangle\right)\\
       \le&     - \frac{1}{2} |\dd \eta|([-\tau,0])| f(0)|^2  
    -   \frac{1}{2}\int_{-\tau } ^0 |f|^2  |\dd\eta|   
            +   \Re  \langle f(0), \int_{-\tau}^0 f \dd\eta\rangle   \\
\le  &      - \frac{1}{2}  |\dd \eta|([-\tau,0])| f(0)|^2   
    -   \frac{1}{2}\int_{-\tau } ^0 |f|^2  |\dd\eta|    
            +   \left(\int_{-\tau}^0 |f(0)|^2 |\dd\eta|\right)^{\half}  \left( \int_{-\tau}^0 |f|^2 |\dd\eta|\right)^{\half}  \\
\le &     - \frac{1}{2}  |\dd \eta|([-\tau,0])| f(0)|^2 
     -\frac{1}{2}\int_{-\tau } ^0 |f|^2  |\dd\eta|    
            +    \frac{1}{2} \int_{-\tau}^0 |f(0)|^2 |\dd\eta|   + \frac{1}{2}\int_{-\tau}^0 |f|^2 |\dd\eta|  \\ 
    \le & \,0.
   \end{align*}
   This concludes the proof thanks to Theorem \ref{thm:pHs2}.
  \end{proof}

 \section{Discrete delay operators}
\label{sec:disc}
In this section we focus on discrete delay operators as in Example \ref{exa:discdelay}. First we compare the setting of this paper to the one in \cite{breiten2024towards}.

\begin{example}[Finitely many point delays]\label{exa:breitencompare}
We consider the following class of finite-dimensional systems
\begin{equation}
    \begin{cases}
        H\dot{z}(t)= \displaystyle (J-R) z(t)-\sum_{i=1}^p Z_i z(t-\tau_i)+Gu(t),
        &t \ge 0,  \\
        z(t)=f(t) ,                        &  t\in \left[-\tau,0 \right],  
        \\
        y(t)= G^\top   z(t) & t\ge 0,  
    \end{cases}
    \label{PHDDEBreiten}
\end{equation}
with Hamiltonian $\mathcal{H}\colon \CC([-\tau,\infty);\mathbb R^n)\times [0,\infty)\rightarrow \mathbb R$ given by
\begin{align}\label{eq:HamBreiten}
\mathcal{H}(z,t)&=\frac{1}{2}z(t)^\top Hz(t) + \sum_{i=1}^p\int_{t-\tau_i}^t z(s)^\top \Theta_i z(s) \dd s,\\
&= \frac{1}{2}z(t)^\top Hz(t) + \sum_{i=1}^p\int_{-\tau_{i}}^{-\tau_{i-1}}  z(t+s)^\top \left(\sum_{j=i}^p \Theta_j\right) z(t+s) \dd s,
\end{align}
where $0=\tau_0<\tau_1<\tau_2<\cdots<\tau_p=\tau $, $Z_i\in \mathbb R^{n\times n}$, $i=1,\cdots, p$,  $J\in \mathbb R^{n\times n}$ is  skew-adjoint and $R\in \mathbb R^{n\times n}$ is positive, self-adjoint.  Further, we assume $G\in \mathbb R^{n\times m}$ and $H, \Theta_1,\cdots,\Theta_p\in \mathbb R^{n\times n}$ to be positive, self-adjoint.

We describe first how this system relates to our formalism.
Multiply first the state equation from the left by $H^{-\half}$ and use  the state space transformation $x(t)=H^{\half}z(t)$. This yields 
\begin{equation*}
    \begin{cases}
        \dot{x}(t)= \displaystyle H^{-\half}(J-R) H^{-\half}x(t)-\sum_{i=1}^p H^{-\half}Z_i H^{-\half}x(t-\tau_i)+H^{-\half} Gu(t)
        &t \ge 0,  \\
       x(t)= H^{-\half} f(t)                        &  t\in \left[-\tau,0 \right],  
        \\
      y(t)= G^\top  H^{-\half} x(t) & t\ge 0,  
    \end{cases}
\end{equation*}
with Hamiltonian $\mathcal{H}\colon \CC([-\tau,\infty);\mathbb R^n)\times [0,\infty)\rightarrow \mathbb R$ given by
\begin{align*}
\mathcal{H}(x,t)&= \frac{1}{2}x(t)^\top x(t) + \sum_{i=1}^p\int_{-\tau_{i}}^{-\tau_{i-1}}  x(t+s)^\top \left(\sum_{j=i}^p H^{-\half}\Theta_jH^{-\half}\right) x(t+s) \dd s,
\end{align*}

Next, we define $Q_0\coloneqq I$, $A_0\coloneqq H^{-\half}( J-R)H^{-\half}$, $A_i \coloneqq -H^{-\half}Z_i H^{-\half}$,  $i=1,\cdots, p,$ $B\coloneqq H^{-\half} G$, 
\begin{equation}
    \Omega(s)\coloneqq \begin{cases} 
    0,&s=-\tau_p,\\
       \displaystyle \sum_{j=i}^{p} H^{-\half}\Theta_j H^{-\half}, & s\in (-\tau_{i},- \tau_{i-1}], \text{ } i=1 ,\ldots, p\\
        
       \end{cases}
\end{equation}

and
  \begin{equation*}
        \Phi f \coloneqq \sum_{i=1}^p A_i f(-\tau_i), \quad f\in \HH^1([-\tau, 0];\mathbb R^n).
    \end{equation*}
We calculate 
\begin{equation*}
    \dd \Omega=\sum_{j=1}^{p} H^{-\half}\Theta_jH^{-\half}\, \delta_{-\tau_j}.
\end{equation*}

Thus, we obtain the delay system \eqref{PHDDE} with Hamiltonian \eqref{eq:Ham} on the state space $\ZO= \mathbb{R}^n  \times \LL^2 \left( [-\tau,0] ; \mathbb{R}^n\right)$.
Theorem \ref{thm:pHs2} implies that system \ref{PHDDEBreiten} is a port-Hamiltonian system if and only if 
    \begin{dmath*}
   \Bigl\langle f(0),  \Bigl( - H^{-\half}R H^{-\half} +\frac{1}{2} \Omega(0) \Bigr)f(0)\Bigr\rangle     - \frac{1}{2}\langle f(-\tau_p),  \Omega(-\tau_p) f(-\tau_p)\rangle   -
            \frac{1}{2}\int_{-\tau_p } ^0 \langle f,(\dd\Omega) f\rangle
            +     \langle f(0),\Phi f\rangle \le 0
 \end{dmath*}
 for every $f \in \HH^{1}( [-\tau_p , 0] ; \mathbb R^n)$. For $f \in \HH^{1}( [-\tau_p , 0] ; \mathbb R^n)$ we calculate 
 \begin{align*}
  \Bigl\langle f(0),  &\left( - H^{-\half}R H^{-\half} +\frac{1}{2} \Omega(0) \right)f(0) \Bigr \rangle   - \frac{1}{2}\langle f(-\tau_p),  \Omega(-\tau_p) f(-\tau_p)\rangle   
-\frac{1}{2}\int_{-\tau_p } ^0 \langle f,(\dd\Omega) f\rangle
            +     \langle f(0),\Phi f\rangle\\
            =&     \langle f(0),  \left( - H^{-\half}R H^{-\half} +\frac{1}{2} \Omega(0) \right)f(0)\rangle \\
            & - \frac{1}{2}\sum_{i=1}^p \langle f(-\tau_i),  H^{-\half}\Theta_j H^{-\half} f(-\tau_i)\rangle +\frac{1}{2}\sum_{i=1}^p  \langle f(0),A_i f(-\tau_i) \rangle +\frac{1}{2}\sum_{i=1}^p  \langle A_i f(-\tau_i) , f(0)\rangle\\
            =& \frac{1}{2}\begin{bmatrix}
        f(0)\\
         f(-\tau_1)\\
        f(-\tau_2)\\
        \vdots\\ f(-\tau_p)
    \end{bmatrix}^{\top} \begin{bmatrix}
         -2 H^{-\half} R  H^{-\half}+\Omega(0)  &     A_1 &   \cdots &  \cdots & A_p\\
           A_1^\top & - H^{-\half}\Theta_1 H^{-\half}  & 0&  \cdots &0\\
          \vdots  & 0 &\ddots 
          & \ddots& \vdots\\
        \vdots  &\vdots & \ddots  & \ddots &0\\ 
        A_p^\top & 0& \cdots & 0 &-H^{-\half}\Theta_p H^{-\half}
    \end{bmatrix}\begin{bmatrix}
        f(0)\\
         f(-\tau_1)\\
        f(-\tau_2)\\
        \vdots\\f(-\tau_p)\end{bmatrix} \\
        =& - \frac{1}{2}\begin{bmatrix}
         H^{-\half}f(0)\\
          H^{-\half}f(-\tau_1)\\
         H^{-\half}f(-\tau_2)\\
        \vdots\\  H^{-\half}f(-\tau_p)
    \end{bmatrix}^{\top} \begin{bmatrix}
         2 R -\sum_{j=1}^{p} \Theta_j  &      Z_1 &    Z_2 &  \cdots &  Z_p\\
           Z_1^\top   &  \Theta_1   & 0&  \cdots &0\\
          Z_2^\top   & 0 &  \Theta_2  & \ddots& \vdots\\
        \vdots  &\vdots & \ddots  & \ddots &0\\ 
        Z_p^\top & 0& \cdots & 0 & \Theta_p 
    \end{bmatrix}\begin{bmatrix}
         H^{-\half}f(0)\\
         H^{-\half} f(-\tau_1)\\
         H^{-\half}f(-\tau_2)\\
        \vdots\\ H^{-\half}f(-\tau_p)\end{bmatrix}
 \end{align*}
 We note (as was done in \cite{breiten2024towards}) that   for every $x\in (\mathbb R^n)^p$ there exists a function $f \in \HH^{1}( [-\tau , 0] ; \mathbb{R}^n)$ with $f(-\tau_k)=x_k$. Thus, the delay system \eqref{PHDDE} with Hamiltonian \eqref{eq:HamBreiten} is a port-Hamiltonian system if and only if 
\begin{equation*}\begin{bmatrix}
         2 R -\sum_{j=1}^{p} \Theta_j  &      Z_1 &    Z_2 &  \cdots &  Z_p\\
           Z_1^\top   &  \Theta_1   & 0&  \cdots &0\\
          Z_2^\top   & 0 &  \Theta_2  & \ddots& \vdots\\
        \vdots  &\vdots & \ddots  & \ddots &0\\ 
        Z_p^\top & 0& \cdots & 0 & \Theta_p 
    \end{bmatrix} \ge 0.\end{equation*}
This equivalence is exactly the result of \cite{breiten2024towards}. Note that this positivity condition was in fact the definition of the port-Hamiltonian property in that work, and the equivalence to the dissipativity inequality (passivity) \eqref{eq:dissi} was the main result there. In any case, we see that our definition is a consistent extension of the one given in \cite{breiten2024towards}.
\end{example}

\begin{remark}
For the delay system in our formalism \eqref{PHDDE}, and with the abbreviation $\Omega_i\coloneqq \Omega(-\tau_{i})$  the previous condition takes the form
\begin{equation*} \begin{bmatrix}
         -2A_0-\Omega_0  &     -A_1 &   -A_2 &  \cdots & -A_p\\
          - A_1^\top & \Omega_0-\Omega_1  & 0&  \cdots &0\\
         - A_2^\top   & 0 & \Omega_1-\Omega_2  & \ddots& \vdots\\
        \vdots  &\vdots & \ddots  & \ddots &0\\ 
       - A_p^\top & 0& \cdots & 0 &\Omega_{p-1}-\Omega_{p} 
    \end{bmatrix} \ge 0.\end{equation*}
    It is plain to see that this condition is equivalent to
     \begin{equation*} \begin{bmatrix}
          -2A_0-\Omega_0  &     A_1 &   A_2 &  \cdots & A_p\\
            A_1^\top & \Omega_0-\Omega_1  & 0&  \cdots &0\\
           A_2^\top   & 0 & \Omega_1-\Omega_2  & \ddots& \vdots\\
         \vdots  &\vdots & \ddots  & \ddots &0\\ 
         A_p^\top & 0& \cdots & 0 &\Omega_{p-1}-\Omega_{p} 
     \end{bmatrix} \ge 0,\end{equation*}
     i.e., the signs of the coefficients appearing in the discrete delay operator are immaterial.
\end{remark}
The next corollary generalizes the characterization of port-Hamiltonian delay systems as in Example \ref{exa:breitencompare}. We begin with some lemmata.
\begin{lemma}\label{lem:matrixlem}
     Let $M=(m_{ij})_{i,j\in\mathbb{N}}$ be an infinite matrix of real numbers such that
     \begin{equation*}
     \sum_{i\in \mathbb N}     \sum_{j\in \mathbb N}|m_{ij}|<\infty.
     \end{equation*}
      Then we have the following:
   \begin{enumerate}[(a)]
   \item  For every bounded sequence  $x=(x_n)_{n\in \mathbb{N}}\in \ell^\infty$ of real numbers $Mx$ exists and $Mx\in \ell^1$.
   \item The following assertions are equivalent:
   \begin{enumerate}[(i)]
\item For every bounded sequence  $x=(x_n)_{n\in \mathbb{N}}\in \ell^\infty$
\begin{equation*}
x^\top Mx\geq 0.
\end{equation*}
\item For every finite sequence  $x=(x_n)_{n\in \mathbb{N}}$ of real numbers 
\begin{equation*}
x^\top Mx\geq 0.
\end{equation*}
\end{enumerate}
     \end{enumerate}
     \end{lemma}
  \begin{proof} (a) is a trivial calculation, but also follows (trivially) from the precise characterization of matrices mapping $\ell^\infty$ into $\ell^1$, see, e.g., \cite[Nr.72]{StieglitzTietz}.

  \medskip\noindent (b) Only the implication (ii)$\Rightarrow$(i) requires a proof. Take a bounded sequence $x$ and $\varepsilon>0$ arbitrarily. From the assumption on the infinite matrix it follows that for some $N\in \mathbb{N}$
  \begin{equation*}
\sum_{\max\{i,j\}\geq N}|m_{ij}|<\varepsilon.
  \end{equation*}
  Set $x^N_k\coloneqq x_k$ for $k\in\{1,\dots,N\}$ and $x^N_k\coloneqq 0$ otherwise.  We thus obtain
\begin{equation*}
x^\top Mx=(x-x^N)^\top Mx+x^{N\top} Mx^N+x^{N\top} M(x-x^N)\geq -\varepsilon \|x\|_\infty+x^{N\top} Mx^N-\varepsilon \|x\|_\infty\geq -2\varepsilon \|x\|_\infty.
     \end{equation*}
This, being true for every $\varepsilon>0$, implies $x^\top Mx\geq 0$.
  \end{proof}

\begin{corollary}\label{cor:pHS3}
Let $(A_k)_{k\in \mathbb N}$ be a sequence in $\RR^{n\times n}$ with $(\|A_k\|)_{k\in \mathbb N}\in \ell^1$, take $0=\tau_0<\tau_1<\tau_2<\cdots<\tau_k<\tau $ with $\tau\coloneqq \lim_{k\rightarrow \infty}\tau_k$, and let the delay operator $\Phi$ be of the form 
    \begin{equation*}
        \Phi f \coloneqq \sum_{k=1}^\infty A_k f(-\tau_k), \quad f\in \HH^1([-\tau, 0];\mathbb{R}^n).
        \end{equation*}
        Suppose $(\Theta_k)_{k\in \mathbb{N}}$ is another sequence of positive self-adjoint operators in  $\mathbb{R}^{n\times n}$. 
       Define
    \begin{equation}\label{eq:omega_mat}
    \Omega (s)= \sum_{i=k+1}^\infty \Theta_i \qquad \text{if}\qquad  -\tau_{k+1}<s\le-\tau_{k}\text{ for some $ k\in \mathbb N_0$, and } \Omega(-\tau):=0.
\end{equation}
Notice that
\[
\Omega(0)=\sum_{k=1}^\infty \Theta_k.
\]
    Then  the following are  equivalent:
\begin{enumerate}[(i)]
    \item The system  \eqref{PHDDE} with Hamiltonian \eqref{eq:Ham} is a port-Hamiltonian delay system.
    \item   For every $f \in \HH^{1}( [-\tau , 0] ; \RR^n)$  we have 
    \begin{equation}\label{eq:pHs_equiv}
      \begin{bmatrix}
        f(0)\\
         f(-\tau_1)\\
        f(-\tau_2)\\
        \vdots\\
    \end{bmatrix}^{\top} \begin{bmatrix}
         -  2A_0 -\Omega_0  &     A_1 &   A_2 &  \cdots & \\
           A_1^\top & \Theta_1 & 0&  \cdots &\\
          A_2^\top   & 0 & \Theta_2   & \ddots\\
        \vdots  &\vdots & \ddots  & \ddots \\ 
    \end{bmatrix}\begin{bmatrix}
        f(0)\\
         f(-\tau_1)\\
        f(-\tau_2)\\
        \vdots\\\end{bmatrix}  \ge 0.
    \end{equation}
    \item   For every finite sequence $(x_k)_{k\in \mathbb N_0}$ in $\mathbb{R}^n$  we have 
  \begin{equation}\label{eq:pHs_equiv3}
      \begin{bmatrix}
        x_0\\
         x_1\\
        x_2\\
        \vdots\\
    \end{bmatrix}^{\top} \begin{bmatrix}
        -  2A_0 -\Omega_0  &     A_1 &   A_2 &  \cdots & \\
           A_1^\top & \Theta_1 & 0&  \cdots &\\
          A_2^\top   & 0 & \Theta_2  & \ddots\\
        \vdots  &\vdots & \ddots  & \ddots \\ 
    \end{bmatrix}\begin{bmatrix}
        x_0\\
         x_1\\
        x_2\\
        \vdots\\\end{bmatrix}  \ge 0.
    \end{equation}
\item   For every bounded sequence $(x_k)_{k\in \mathbb N_0}$ in $\RR^n$  we have 
  \begin{equation}\label{eq:pHs_equivbdd}
     \begin{bmatrix}
        x_0\\
         x_1\\
        x_2\\
        \vdots\\
    \end{bmatrix}^{\top} \begin{bmatrix}
        -  2A_0 -\Omega_0  &     A_1 &   A_2 &  \cdots & \\
           A_1^\top & \Theta_1 & 0&  \cdots &\\
          A_2^\top   & 0 & \Theta_2  & \ddots\\
        \vdots  &\vdots & \ddots  & \ddots \\ 
    \end{bmatrix}\begin{bmatrix}
        x_0\\
         x_1\\
        x_2\\
        \vdots\\\end{bmatrix}  \ge 0.
    \end{equation}
\item For each $N\in \mathbb{N}$ and for each $x_0,x_1,\dots,x_N\in \mathbb{R}^n$ we have
\begin{equation}\label{eq:pHs_equiv2}
     \begin{bmatrix}
        x_0\\
         x_1\\
        \vdots\\x_N
    \end{bmatrix}^{\top} \begin{bmatrix}
        -  2A_0 -\Omega_0  &     A_1 &   \cdots&  A_N \\
           A_1^\top & \Theta_1 & 0&  \cdots &\\
        \vdots   & 0 & \ddots  & \ddots\\
        A_N^\top  &\vdots & \ddots  & \Theta_N 
    \end{bmatrix}\begin{bmatrix}
        x_0\\
         x_1\\
        \vdots\\x_N\end{bmatrix}  \ge 0.
    \end{equation}
 \end{enumerate}
 \end{corollary}
\begin{proof} The equivalence of Assertions (iii), and (v) and that (iv) implies (iii) are trivial.
For $f \in \HH^{1}( [-\tau , 0] ; \mathbb{R}^n)$   we calculate
\begin{align*}
  f(0)^\top  &\left(  A_0 +\frac{1}{2} \Omega(0)\right)f(0)   
              -
            \frac{1}{2}\int_{-\tau } ^0 f^\top\dd\Omega f-\frac12f(-\tau)^\top\Omega(-\tau)f(-\tau)
            +      f(0)^\top \Phi f \\ 
= &      f(0)^\top  \left(  A_0 +\frac{1}{2} \Omega(0) \right)f(0)  
    -   \frac{1}{2}\sum_{k=1 } ^\infty  f(-\tau_k)^\top\Omega_k f(-\tau_k)  
            +     f(0)^\top \sum_{k=1}^\infty A_k f(-\tau_k),  
   \end{align*}
   which, by Theorem \ref{thm:pHs2}, implies the equivalence of Assertions (i) and (ii). That Assertion (ii) implies Assertion (iii) follows as in Example \ref{exa:breitencompare} by finding suitable $\HH^1$-functions with prescribed values. From Assertion (iv) we infer Assertion (ii) easily, for  an $\HH^1$-function is bounded. The remaining implication (iii)$\Rightarrow$(iv) follows from  Lemma \ref{lem:matrixlem}.
     \end{proof}  

\begin{remark}
Exactly the same proof yields the following variant of the previous characterization for the infinite dimensional situation (we just repeat two of the equivalent conditions). Let $H$ be a Hilbert space, let $A_0$ be maximal dissipative on $H$, let $(A_k)_{k\in \mathbb N}$ be a sequence in $\LLL(H)$ with $(\|A_k\|)_{k\in \mathbb N}\in \ell^1$, take $0=\tau_0<\tau_1<\tau_2<\cdots<\tau_k<\cdots<\tau $ with $\tau\coloneqq \lim_{k\rightarrow \infty}\tau_k$, and let the delay operator $\Phi$ be of the form 
    \begin{equation*}
        \Phi f \coloneqq \sum_{k=1}^\infty A_k f(-\tau_k), \quad f\in \HH^1([-\tau, 0];H).
        \end{equation*}
        Suppose $(\Theta_k)_{k\in \mathbb{N}}$ is a sequence of positive self-adjoint operators in $\LLL(H)$, and 
       define
    \begin{dmath}\label{eq:omega_Op}
    \Omega (s)= \sum_{i=k+1}^\infty \Theta_i  \text{ if }  {-\tau_{k+1}<s\le-\tau_{k}} \text{ for some }  {k\in \mathbb N_0,}
    \text{ and } \Omega(-\tau){:=}0.
\end{dmath}
    Then  the following are  equivalent:
\begin{enumerate}[(i)]
\item The system  \eqref{PHDDE} with Hamiltonian \eqref{eq:Ham} is a port-Hamiltonian delay system.
\item For each $N\in \mathbb{N}$, for each $x_0\in D(A_0)$ and for every $x_1,\dots,x_N\in H$ we have
\begin{equation}\label{eq:pHs_equivinf}
   \Re  \left\langle\begin{bmatrix}
        x_0\\
         x_1\\
        \vdots\\x_N
    \end{bmatrix}, \begin{bmatrix}
        -  2A_0 -\Omega_0  &     A_1 &   \cdots&A_N \\
           A_1^* & \Theta_1 & 0&  \cdots &\\
        \vdots   & 0 & \ddots  & \ddots\\
        A_N^*  &\vdots & \ddots  & \Theta_N 
    \end{bmatrix}\begin{bmatrix}
        x_0\\
         x_1\\
        \vdots\\x_N\end{bmatrix}\right\rangle \ge 0.
    \end{equation}
    The scalar product is in $H\times \cdots \times H$ ($N+1$ copies).
\end{enumerate}
\end{remark}

\noindent Finally, we present a sufficient condition for the port-Hamiltonian property of delay systems with a discrete delay operator.
\begin{corollary}
Let $(A_k)_{k\in \mathbb N}$ be a sequence in $\RR^{n\times n}$ with $(\|A_k\|)_{k\in \mathbb N}\in \ell^1$, take $0=\tau_0<\tau_1<\tau_2<\cdots<\tau_k<\tau $ with $\tau\coloneqq \lim_{k\rightarrow \infty}\tau_k$, and let the delay operator $\Phi$ be of the form 
    \begin{equation*}
        \Phi f \coloneqq \sum_{k=1}^\infty A_k f(-\tau_k), \quad f\in \HH^1([-\tau, 0];\mathbb{R}^n).
        \end{equation*}
       
Suppose 
    $x^\top A_0 x \le - \|(A_k)_{k\in\mathbb{N}}\|_{\ell^1}\|x\|^2$ for every $x\in \mathbb R^n$, and define
    \begin{equation}\label{eq:Omega2}
    \Omega (s)= \sum_{i=k+1}^\infty \|A_i\|I \qquad \text{if}\qquad  -\tau_{k+1}<s\le-\tau_{k}\text{ for some $ k\in \mathbb N_0$, and } \Omega(-\tau)\coloneqq 0.
\end{equation}
    Then the system \eqref{PHDDE} is a port-Hamiltonian system with Hamiltonian \eqref{eq:Ham}.
\end{corollary}
\begin{proof}
For $f \in \HH^{1}( [-\tau , 0] ; \mathbb{R}^n)$   we calculate, by inserting the definitions and by applying the Cauchy-Schwarz inequality once, that
\begin{align*}
  f(0)^\top  &\left(  A_0 +\frac{1}{2} \Omega(0) \right)f(0)   
              -
            \frac{1}{2}\int_{-\tau } ^0 f^\top \dd \Omega f
            +      f(0)^\top \Phi f-f(-\tau)^\top \Omega(-\tau)f(-\tau) \\ 
= & f(0)^\top  \left(  A_0 +\frac{1}{2} \Omega(0) \right)f(0)   
              -
            \frac{1}{2}\int_{-\tau } ^0 f^\top \dd \Omega f
            +      f(0)^\top \Phi f \\
\le &     - \frac{1}{2}\|(A_k)_{k\in \mathbb{N}}\|_{\ell^1}\|f(0)\|^2
    -   \frac{1}{2}\sum_{k=1 } ^\infty  \|A_k\|\|f(-\tau_k)\|^2   
            +    f(0)^\top \sum_{k=1}^\infty A_k f(-\tau_k)\\
\le &     - \frac{1}{2}\|(A_k)_{k\in \mathbb{N}}\|_{\ell^1}\|f(0)\|^2 
    -   \frac{1}{2}\sum_{k=1 } ^\infty \|A_k\| |f(-\tau_k)|^2   
            +    \left(\sum_{k=1}^\infty \|A_k\| \|f(0)\|^2\right)^{\half}  \left(\sum_{k=1}^\infty \|A_k\| \|f(-\tau_k)\|^2\right)^{\half}\\
            \le &     - \frac{1}{2}\|(A_k)_{k\in \mathbb{N}}\|_{\ell^1}\|f(0)\|^2 
    -   \frac{1}{2}\sum_{k=1 } ^\infty \|A_k\| |f(-\tau_k)|^2   
            +        \frac{1}{2} \sum_{k=1}^\infty \|A_k\| \|f(0)\|^2  +   \frac{1}{2}\sum_{k=1}^\infty \|A_k\| \|f(-\tau_k)\|^2\\
\le &       
    -   \frac{1}{2}\sum_{k=1 } ^\infty \|A_k\| \|f(-\tau_k)\|^2   
              +   \frac{1}{2}\sum_{k=1}^\infty \|A_k\| \|f(-\tau_k)\|^2\\
     = &0.
     \end{align*} 
Finally,  Theorem \ref{thm:pHs2} implies the statement of the theorem. 
\end{proof}

\section{Examples}
To showcase the versatility of our system class and also the simplicity of the sufficient conditions in Theorem \ref{thm:webbmini}, we now briefly discuss some examples. 

\subsection{Diffusion-equation with delay}

Consider a bounded domain $\Omega\subseteq \mathbb{R}^d$ with piecewise $\CC^1$-boundary in $\mathbb{R}^d$, the Hilbert space $H=\LL^2(\Omega)$ and the Laplace operator $\Delta$ with homogeneous Dirichlet boundary conditions, i.e., with domain
\[
D(\Delta)=\HH^1_0(\Omega)\cap \HH^2(\Omega).
\]
We consider the delay equation
\begin{align*}
\dot x(t)&=\Delta x(t)+iVx(t)+\Phi x_t+Bu(t),\\
y(t)&=B^*x(t),
\end{align*}
where $V:\Omega\to \mathbb{R}$ is suitably chosen  real potential and $\Phi$ is a delay operator of the sort considered in this paper. Our aim is to find conditions for which this system becomes  port-Hamiltonian.

First,  we study basic conditions on $V$ in order that the multiplication operator by $V$ becomes relatively $\Delta$-bounded.  The arising assumptions will depend on the dimension $d$, as is usual and elementarily obtained from Sobolev embeddings.

Let $V\colon \Omega\to\RR$ be a measurable function and consider the multiplication operator  $f\mapsto Vf$, denoted also by $V$, with maximal domain $D(V)\coloneqq \{f\in \LL^2(\Omega):Vf\in \LL^2(\Omega)\}$. 
Suppose first that  $d\in \{1,2,3\}$. We then have the Sobolev embedding $\HH^2(\Omega)\hookrightarrow \LL^{\infty}(\Omega)$, so if $V\in \LL^2(\Omega)$, then
\[
\|Vf\|_2\leq \|V\|_{2}\|f\|_{\infty}\leq C_\Omega\|V\|_{2}\|f\|_{\HH^2}\quad\text{for every $f\in D(\Delta)$},
\]
showing that  $D(\Delta)\subseteq D(V)$ and that $V$ is relatively $\Delta$-bounded.

Next, let $d=4$, let $V\in \LL^{s}(\Omega)$ with $s\in (2,\infty]$, and take $r$ satisfying $\frac 1r=\frac1{s}-\frac 12=\frac1{s}-\frac{2}{d}$. Then  we have the  Sobolev embedding $\HH^2(\Omega)\hookrightarrow \LL^{r}(\Omega)$, so by  H\"older's inequality we conclude
\[
\|Vf\|_2\leq \|V\|_{s}\|f\|_{r}\leq C_\Omega\|V\|_{s}\|f\|_{\HH^2}\quad\text{for every $f\in D(\Delta)$},
\]
showing again that  $D(\Delta)\subseteq D(V)$ and that $V$ is relatively $\Delta$-bounded.

Finally, if $d>4$ and $V\in \LL^{d/2}(\Omega)$, then the Sobolev embedding $\HH^2(\Omega)\hookrightarrow \LL^{r}(\Omega)$ with $\frac{1}{r}=\frac 12-\frac2d$ and H\"older's inequality yield
\[
\|Vf\|_2\leq \|V\|_{d/2}\|f\|_{r}\leq C_\Omega\|V\|_{d/2}\|f\|_{\HH^2}\quad\text{for every $f\in D(\Delta)$},
\]
showing again that  $D(\Delta)\subseteq D(V)$ and that $V$ is relatively $\Delta$-bounded.

Since, in any of the above cases,  $\Delta$ and $iV$ are both dissipative, by Corollary III.2.9 in \cite{engel2000one} we conclude that the  closure $A_0$ of $iV+\Delta$  generates a contraction semigroup on $\LL^2(\Omega)$. (Note that $A_0$, formally, has the form $J-R$ with $J$ skew-adjoint and $R$ strictly positive, self-adjoint.)  Denote by $\lambda_0>0$ the smallest eigenvalue of $-\Delta$ (so $-\lambda_0$ is the growth bound of the Dirichlet heat semigroup generated by $\Delta$). We have
\[
\Re\langle A_0 f,f\rangle\leq -\lambda_0\|f\|_2^2.
\]
Let $\tau>0$ and let $\eta\colon [-\tau,0]\to \LLL(H)$ of bounded variation satisfying $|\dd\eta|([-\tau,0])\leq \lambda_0$, but otherwise arbitrary. From Theorem \ref{thm:webbmini} we immediately obtain that the delay system \eqref{PHDDE} will be port-Hamiltonian with a suitable choice of a Hamiltonian.

\subsection{Distributed-memory car-following model}
We now want to investigate a certain subclass of our system that is particularly interesting in applications.
Consider 
\begin{align*}
    \dot x(t) = -\alpha  x(t) + \beta \int_{-\infty}^0 x_t (s) k(s) \mathrm ds
\end{align*}
where $\alpha, \beta$ are constants and $k$ is a kernel.

This kind of systems occur particularly after linearisation in distributed-memory car-following models (see, e.g., \cite{CheLiuWanLiuFan26,MarSteBacKis25,SipAtaNic07}), compartment model of hematopoiesis (see, e.g., \cite{BerCra15}) or neutrophil dynamics (see, e.g., \cite{BerBelMac01}). For more information on the applications of delay equations we refer to \cite{Ern09} or \cite{Smi11}.

To fit our class, we assume the kernel $k$ to have compact support in $[-\tau,0]$ for some $\tau>0$. Typical examples are the uniform distribution
\begin{align*}
k(s) = \begin{cases}
    1/ \tau &\text{if } -\tau<s<0,\\
    0 &\text{otherwise,}
\end{cases}
\end{align*}
the Epanechnikov-kernel (see \cite{Epa69}), the triangular kernel or a truncated Gaussian kernel. 
So in this case the delay operator has the form
\[
\Phi f=\int_{-\tau}^0 f(s)k(s)\mathrm ds.
\]
All of these kernels are usually normalized, therefore 
\[|\mathrm d\eta | ([-\tau,0]) = \int_{-\tau}^0k(s) \mathrm ds = 1.\]
Thus, Theorem \ref{thm:webbmini} yields that the system is port-Hamiltonian (without input or output) if $-\alpha \le -|\beta|$ or equivalently $\alpha \ge |\beta|$.
This condition is also sufficient for stability of the delay equation, see e.g., \cite{BerBelMac01}.
Similar results hold for the point delay case, i.e.
\begin{align*}
    \dot x(t) = -\alpha x(t) -  \sum_{i \in \mathbb N} \beta_i x_t (s)
\end{align*}
or delay equations with a mixture of distributed and point delays.

\section{AI Disclosure}
No AI tools of any sort have been used either for undertaking the research described in this paper or for the preparation of the manuscript.

\printbibliography
\end{document}